\documentclass[10pt]{amsart}
\usepackage{amsmath}
\usepackage{amssymb} 

\def\beqnn{\begin{eqnarray*}}\def\eeqnn{\end{eqnarray*}}
\newtheorem{theorem}{Theorem}[section]
\newtheorem{lemma}[theorem]{Lemma}
\newtheorem{proposition}[theorem]{Proposition}

\theoremstyle{definition}

\newtheorem{remark}[theorem]{Remark}

\theoremstyle{question}

\numberwithin{equation}{section}

\begin{document}

\title{Hilbert-type operators acting between weighted Fock spaces}


\author{Jianjun Jin}
\address{School of Mathematics Sciences, Hefei University of Technology, Xuancheng Campus, Xuancheng 242000, P.R.China}
\email{jin@hfut.edu.cn, jinjjhb@163.com}

 \author{Shuan Tang}
  \address{School of Mathematics Sciences, Guizhou Normal University, Guiyang 550001, P.R.China} \email{tsa@gznu.edu.cn}

\author{Xiaogao Feng}
  \address{College of Mathematics and Information, China West Normal University, Nanchong 637009, P.R.China}
 \email{fengxiaogao603@163.com}

\thanks{The first author supported by National Natural Science
Foundation of China (Grant Nos. 11501157). The second author was supported by National Natural Science Foundation of China (Grant Nos. 12061022) and the foundation of Guizhou Provincial Science and Technology Department (Grant Nos. [2017]7337 and [2017]5726).
The third author was supported by National Natural Science Foundation of China (Grant
Nos. 11701459).
   }

\thanks{ {\bf{Data Availability Statement:} the datasets generated during and/or analysed during the current study are available from the corresponding author on reasonable request.}
}






\begin{abstract}
In this paper we introduce and study several new Hilbert-type operators acting between the weighted Fock spaces. We provide some sufficient and necessary conditions for the boundedness and compactness  of certain Hilbert-type operators from one weighted Fock space to another. \end{abstract}

\maketitle
\section{Introduction and main results}

In this paper we will use $C, C_1, C_2, \cdots$ to denote universal positive constants that might change from one line to another.  For two positive numbers $A, B$, we write
$A \preceq B$, or $A \succeq B$, if there exists a positive constant $C$ independent of $A$ and $B$ such that $A \leq C B$, or $A \geq C B$, respectively. We will write $A \asymp B$ if both $A \preceq B$ and $A \succeq B$.

Let $\mathbb{C}$ be the complex plane,  $\Omega$  be an open domain in $\mathbb{C}$. Let $\mathcal{H}(\Omega)$ be the class of all holomorphic functions on $\Omega$. In particular, $\mathcal{H}(\mathbb{C})$ is the class of all entire functions.

The {\em Fock space}, denoted by $\mathcal{F}^{2}$, is defined as
$$\mathcal{F}^{2}=\{f\in \mathcal{H}(\mathbb{C}): \|f\|_2^2:=\frac{1}{\pi}\int_{\mathbb{C}}|f(z)e^{-\frac{1}{2}|z|^2}|^2dA(z)<\infty\},$$
where $dA$ is the Lebesgue area measure, see \cite{Zhu}.

For an entire function $f=\sum_{n=0}^{\infty}a_nz^n$.  A  calculation yields that
$$\|f\|_2^2=2\int_{0}^{\infty}(\sum_{n=0}^{\infty}|a_n|^2r^{2n})e^{-r^2}rdr=\sum_{n=0}^{\infty}|a_n|^2 n!.$$
Thus we have $f\in \mathcal{F}^2$ if and only if $\sum_{n=0}^{\infty}|a_n|^2 n!<\infty.$

For $\theta>0$, $\alpha\in \mathbb{R}$, we introduce the {\em weighted Fock space}, which is denoted by $\mathcal{F}^{2}_{\theta, \alpha}$ and defined as
$$\mathcal{F}^{2}_{\theta, \alpha}=\{f=\sum_{n=0}^{\infty} a_n z^n\in \mathcal{H}(\mathbb{C}): \|f\|_{\theta, \alpha}^2:=\sum_{n=0}^{\infty}(n+\theta)^{\alpha}|a_n|^2n!<\infty\}.$$

\begin{remark} When $\alpha=0$,  $\mathcal{F}^{2}_{\theta, \alpha}$ becomes the Fock space $\mathcal{F}^{2}$.  When $\alpha\in \mathbb{N}$,  we know from \cite{CZ} that an entire function $f(z)$ belongs to the Fock-Sobolev space $F^{2, \alpha}$ on $\mathbb{C}$, if and only if $z^{\alpha}f(z)\in \mathcal{F}^{2}$. Then it is easy to see  that $\mathcal{F}^{2}_{\theta, \alpha}$ coincides with the Fock-Sobolev space $F^{2, \alpha}$ on $\mathbb{C}$ for $\alpha\in \mathbb{N}$.
 \end{remark}
 
\begin{remark}
We notice that, if the coefficients of a power series $f(z)=\sum_{n=0}^{\infty} a_n z^n$ satisfy the condition $ \sum_{n=0}^{\infty}(n+\theta)^{\alpha}|a_n|^2n!<\infty$, then $f(z)$
is an entire function.
 \end{remark}

 \begin{remark}
 If  $f(z)=\sum_{n=0}^{\infty} a_n z^n \in \mathcal{F}^{2}_{\theta, \alpha}$, then, by Cauchy-Schwartz inequality, we see that  there is a positive constant $C$, depends only on $\theta, \alpha$, such that
  \begin{equation}|f(z)|\leq  \left[\sum_{n=0}^{\infty}(n+\theta)^{\alpha}|a_n|^2n!\right]^{\frac{1}{2}}\left [\sum_{n=0}^{\infty}(n+\theta)^{-\alpha}\frac{|z|^{2n}}{n!}\right]^{\frac{1}{2}}\leq e^{C|z|^2} \|f\|_{\theta, \alpha}, \nonumber \end{equation}
for all $z\in \mathbb{C}$.
   \end{remark}

It should be pointed out that the weighted Fock space $\mathcal{F}^{2}_{\theta, \alpha}$  can be considered as a special case of the weighted Hardy spaces, which have been detailed studied in \cite{Sh}.  We recall the  definition of the weighted Hardy spaces.  Let $\beta=\{\beta_n\}_{n=0}^{\infty}$ be a sequence of positive numbers. The weighted Hardy space $\mathcal{H}^2(\beta)$ is the class of formal complex power series $f(z)=\sum_{n=0}^{\infty}a_nz^n$ such that
$$\sum_{n=0}^{\infty}|a_n|^2 \beta_n^2<\infty.$$

From \cite{Sh}, we know that,  if $\beta_{n+1}/\beta_n \rightarrow 1$, when $n\rightarrow \infty,$  $\mathcal{H}^2(\beta)$ is a Hilbert space of holomorphic functions on the unit disk $\mathbb{D}$ with the inner product $$\langle f,g\rangle=\sum_{n=0}^{\infty}a_n \overline{b}_{n} \beta_n^2,$$
where $f=\sum_{n=0}^{\infty}a_nz^n$ and $g=\sum_{n=0}^{\infty}b_nz^n$.   In particular, if  we take $\beta_n=1$ for all $n\in \mathbb{N}_0=\mathbb{N}\cup\{0\}$, then $\mathcal{H}^2(\beta)$ reduces the classic Hardy space $H^2(\mathbb{D})$,  which is  a
Hilbert space of holomorphic functions on the unit disk $\mathbb{D}$ with the inner product $\langle f,g\rangle=\sum_{n=0}^{\infty}a_n \overline{b}_{n},$ see \cite{Zh}.

\begin{remark}Define the inner product on $\mathcal{F}_{\theta, \alpha}^{2}$ as
$$\langle f,g\rangle=\sum_{n=0}^{\infty}(n+\theta)^{\alpha}a_n \overline{b}_{n} n!,$$
where $f=\sum_{n=0}^{\infty}a_nz^n$ and $g=\sum_{n=0}^{\infty}b_nz^n.$  Consider the operator $W$ on $\mathcal{F}_{\theta, \alpha}^{2}$, defined as
 $$Wf(z)=\sum_{n=0}^{\infty} [(n+\theta)^{\frac{\alpha}{2}}\sqrt{n!}]a_nz^n, \quad \forall f(z)=\sum_{n=0}^{\infty}a_n z^n \in \mathcal{F}_{\theta, \alpha}^{2}$$

Then we conclude that the operator $W$ on $\mathcal{F}_{\theta, \alpha}^{2}$ is an isometric isomorphism from $\mathcal{F}_{\theta, \alpha}^{2}$ to the Hardy space $H^2(\mathbb{D})$. Consequently, we see that $\mathcal{F}_{\theta, \alpha}^{2}$ is a Hilbert space with the above inner product.  For any nonnegative integer $n$, let $e_n(z)=(n!)^{-\frac{1}{2}}(n+\theta)^{-\frac{\alpha}{2}}z^n.$  Then we see that the set $\{e_n\}$ is an orthonormal basis of $\mathcal{F}_{\theta, \alpha}^{2}$.  By the standard theory of reproducing kernel Hilbert spaces(see Theorem 2.4 in \cite{PR1}), we get that the reproducing kernel $K_{\theta, \alpha}(z,w)$ for $\mathcal{F}_{\theta, \alpha}^{2}$ is
 $$K_{\theta, \alpha}(z,w)=\sum_{n=0}^{\infty} e_n(z) \overline{e_n(w)}=\sum_{n=0}^{\infty}(n+\theta)^{-\alpha}\frac{(z\overline{w})^n}{n!}.$$

In particular, when $\alpha=0$, we have $K_{\theta, 0}(z,w)=e^{z\overline{w}}$; when $\theta\in \mathbb{N}, \alpha=-1$, we have $K_{\theta, -1}(z,w)=e^{z\overline{w}}(z\overline{w}+\alpha).$
\end{remark}


For $f=\sum_{k=0}^{\infty}a_k z^k \in \mathcal{H}(\Omega)$, $\Omega$ is an open domain  in $\mathbb{C}$. The {\em Hilbert operator} $\mathbf{H}$, which is an operator on spaces of holomorphic functions by its action on the Taylor coefficients, is defined as,
\begin{equation*} \mathbf{H}(f)(z):=\sum_{n=0}^\infty \left(\sum_{k=0}^{\infty} \frac{a_k}{k+n+1} \right)z^n.
\end{equation*}

During the last two decades, the Hilbert operator $\mathbf{H}$ and its generalizations defined on various spaces of holomorphic
functions on the unit disk $\mathbb{D}$ in $\mathbb{C}$ have been much investigated.  See, for example, [1-14], \cite{PR},  \cite{YZ1}, \cite{YZ2}.

For $\lambda>0$,  $f=\sum_{k=0}^{\infty}a_k z^k \in \mathcal{H}(\mathbb{C})$, we define the following {\em Hilbert-type operator $ \mathbf{H}_{\lambda}$}  as
\begin{equation*} \mathbf{H}_{\lambda}(f)(z):=\sum_{n=0}^\infty \left[\frac{1}{\sqrt{n!}}\sum_{k=0}^{\infty} \frac{a_k\sqrt{k!}}{(k+n)^{\lambda}+1} \right]z^n.
\end{equation*}

In this paper, we first study the boundedness of $\mathbf{H}_{\lambda}$ acting from one weighted Fock space to another,  and obtain that
\begin{theorem}\label{main-1}
Let $\lambda>0, -1<\alpha, \beta<1$. Then, for any $\theta>0$,  $ \mathbf{H}_{\lambda}$ is bounded from $\mathcal{F}_{\theta, \alpha}^{2}$ to $\mathcal{F}_{\theta, \beta}^{2}$ if and only if $\lambda\geq 1+\frac{1}{2}(\beta-\alpha)$ .
\end{theorem}

Next we consider the case $\lambda=1$. We see from Theorem \ref{main-1} that $ \mathbf{H}_{1}$ is not bounded from $\mathcal{F}_{\theta, \alpha}^{2}$ to $\mathcal{F}_{\theta, \beta}^{2}$ if  $\alpha<\beta$.   We notice that
$$\frac{1}{k+n+1}=\int_{0}^1t^{k+n}dt,\: k, n\in \mathbb{N}_0.$$

To make $ \mathbf{H}_{1}$ bounded from $\mathcal{F}_{\theta, \alpha}^{2}$ to $\mathcal{F}_{\theta, \beta}^{2}$ when $\alpha<\beta$.  For $f=\sum_{n=0}^{\infty}a_n z^n \in \mathcal{H}(\mathbb{C})$,  let $\mu$ be a positive Bore measure on $[0, 1)$,  we introduce the {\em Hilbert-type operator $\mathbf{H}_{\mu}$}, which is defined as
\begin{equation*} \mathbf{H}_{\mu}(f)(z):=\sum_{n=0}^\infty \left(\frac{1}{\sqrt{n!}}\sum_{k=0}^{\infty}\mu[k+n] a_k\sqrt{k!}\right)z^n.
\end{equation*}

Where
\begin{equation}\label{mu}\mu{[n]}=\int_{0}^1 t^{n}d\mu(t),\: n\in \mathbb{N}_{0}.\end{equation}

We then study the problem of characterizing measures $\mu$ such that $\mathbf{H}_{\mu}: \mathcal{F}_{\theta, \alpha}^{2}\rightarrow \mathcal{F}_{\theta, \beta}^{2}$ is bounded and prove that

\begin{theorem}\label{main-2}
Let $-1<\alpha, \beta<1$,  $\mu$ be a finite positive Bore measure on $[0, 1)$ and $\mathbf{H}_{\mu}$ be as above. Then, for any $\theta>0$,   $\mathbf{H}_{\mu}: \mathcal{F}_{\theta, \alpha}^{2}\rightarrow \mathcal{F}_{\theta, \beta}^{2}$ is bounded if and only if $\mu$ is a $[1+\frac{1}{2}(\beta-\alpha)]$-Carleson measure on $[0, 1)$.
\end{theorem}
Here,  for $s>0$, a positive Borel measure $\mu$ on $[0,1)$, we say $\mu$ is an $s$-Carleson measure if there is a constant $C_1>0$ such that $$\mu([t, 1))\leq C_1 (1-t)^s$$ holds for all $t\in [0, 1)$.

Also, we characterize measures $\mu$ such that $\mathbf{H}_{\mu}: \mathcal{F}_{\theta, \alpha}^{2}\rightarrow \mathcal{F}_{\theta, \beta}^{2}$ is compact and shall show that
 \begin{theorem}\label{main-3}
Let $-1<\alpha, \beta<1$,  $\mu$ be a finite positive Bore measure on $[0, 1)$ and $\mathbf{H}_{\mu}$ be as above. Then, for any $\theta>0$,   $\mathbf{H}_{\mu}: \mathcal{F}_{\theta, \alpha}^{2}\rightarrow \mathcal{F}_{\theta, \beta}^{2}$ is compact if and only if $\mu$ is a vanishing $[1+\frac{1}{2}(\beta-\alpha)]$-Carleson measure on $[0, 1)$.
\end{theorem}

Here, an $s$-Carleson measure $\mu$ on $[0, 1)$  is  said to be a vanishing $s$-Carleson measure, if  it satisfies further that
$$\lim_{t\rightarrow 1^{-}}\frac{\mu([t, 1))}{(1-t)^s}=0.$$

The paper is organized as follows. The proofs of Theorem \ref{main-1},  \ref{main-2} and \ref{main-3} will be given in Section 3, 4 and 5, respectively. Some lemmas will be proved in Section 2. Several remarks will be finally presented in the last section.

\section{Some lemmas}
In this section, we establish some lemmas, which will be used in the proofs of Theorem \ref{main-1}, \ref{main-2} and \ref{main-3} .
\begin{lemma}\label{lem-0}
For $\theta>0, -1<\alpha, \beta<1$, we define
\begin{equation*}
w_{\alpha, \beta}^{[1]}(n):=\sum_{k=0}^{\infty}\frac{1}{(k+n+2\theta)^{1+\frac{1}{2}(\beta-\alpha)}}\cdot \frac{(n+\theta)^{\frac{1-\beta}{2}}}{(k+\theta)^{\frac{1+\alpha}{2}}},\: n\in \mathbb{N}_{0},
\end{equation*}
and
\begin{equation*}
w_{\alpha, \beta}^{[2]}(k):=\sum_{n=0}^{\infty}\frac{1}{(k+n+2\theta)^{1+\frac{1}{2}(\beta-\alpha)}}\cdot \frac{(k+\theta)^{\frac{1+\alpha}{2}}}{(n+\theta)^{\frac{1-\beta}{2}}},\: k\in \mathbb{N}_0.
\end{equation*}
Then \begin{equation}\label{w-1}
w_{\alpha, \beta}^{[1]}(n)\leq B(\frac{1+\beta}{2}, \frac{1-\alpha}{2})(n+\theta)^{-\beta}, n\in \mathbb{N}_{0};
\end{equation}
\begin{equation}\label{w-2}
w_{\alpha, \beta}^{[2]}(k)\leq B(\frac{1+\beta}{2}, \frac{1-\alpha}{2})(k+\theta)^{\alpha}, k\in \mathbb{N}_0.
\end{equation}
\end{lemma}
Here $B(\cdot, \cdot)$ is the Beta function, defined as
$$B(u,v)=\int_{0}^{\infty}\frac{t^{u-1}}{(1+t)^{u+v}}\,dt,\: u>0,v>0.$$
It is known that
$$B(u,v)=\int_{0}^{1}t^{u-1}(1-t)^{v-1}\,dt=\frac{\Gamma(u)\Gamma{(v)}}{\Gamma(u+v)}. $$
and $B(u,v)=B(v,u)$, where $\Gamma(x) $ is the Gamma function, defined as
$$\Gamma(x)=\int_{0}^{\infty}e^{-t} t^{x-1}\,dt,\: x>0.$$
For more detailed introduction to the Beta function and Gamma function, see \cite{W} .
\begin{proof}[Proof of Lemma \ref{lem-0} ]
Since $-1<\alpha, \beta<1$, we see that
\begin{eqnarray}
w_{\alpha, \beta}^{[1]}(n)\leq  \int_{0}^{\infty}\frac{1}{(x+n+\theta)^{1+\frac{1}{2}(\beta-\alpha)}}\cdot \frac{(n+\theta)^{\frac{1-\beta}{2}}}{x^{\frac{1+\alpha}{2}}}\,dx
= B(\frac{1+\beta}{2}, \frac{1-\alpha}{2})(n+\theta)^{-\beta}.\nonumber
\end{eqnarray}

Similarly, we can obtain that
\begin{equation*}
w_{\alpha, \beta}^{[2]}(k)\leq B(\frac{1+\beta}{2}, \frac{1-\alpha}{2})(k+\theta)^{\alpha}.
\end{equation*}
Lemma \ref{lem-0} is proved.
\end{proof}

\begin{lemma}\label{lem-1}
Let $\theta>0, -1<\alpha, \beta<1$. Let $\mu$ be a positive Borel measure on $[0, 1)$ and $\mu[n]$ is defined as in (\ref{mu}) for $n\in \mathbb{N}_0$. If $\mu$ is a  $[1+\frac{1}{2}(\beta-\alpha)]$-Carleson measure on $[0, 1)$, then \begin{equation}\label{mun-1}\mu[n] \preceq \frac{1}{(n+2\theta)^{1+\frac{1}{2}(\beta-\alpha)}}\end{equation}
holds for all $n \in \mathbb{N}_0$. Furthermore, if $\mu$ is a  vanishing $[1+\frac{1}{2}(\beta-\alpha)]$-Carleson measure on $[0, 1)$, then
 \begin{equation}\label{mun-2}\mu[n] =o\left(\frac{1}{(n+2\theta)^{1+\frac{1}{2}(\beta-\alpha)}}\right), \quad n\rightarrow \infty.\end{equation}
\end{lemma}

\begin{proof}
For $n\in \mathbb{N}$, we get from integration by parts that
\begin{eqnarray}
\mu[n]=\int_{0}^1 t^{n}d\mu(t)&=&\mu([0,1))-n\int_{0}^1 t^{n-1}\mu([0, t))dt \nonumber \\
&=& n\int_{0}^1 t^{n-1}\mu([t, 1))dt.\nonumber
\end{eqnarray}

If $\mu$ is a $[1+\frac{1}{2}(\beta-\alpha)]$-Carleson measure on $[0, 1)$, then we see that there is a constant $C_3>0$ such that
$$\mu([t,1))\leq C_3 (1-t)^{1+\frac{1}{2}(\beta-\alpha)}$$
holds for all $t\in [0,1)$. It follows that
\begin{eqnarray}\mu[n] &\leq & C_3 n\int_{0}^1 t^{n-1}(1-t)^{1+\frac{1}{2}(\beta-\alpha)}dt\nonumber \\&=&C_3 \frac{n\Gamma(n)\Gamma(2+\frac{1}{2}(\beta-\alpha))}{\Gamma(n+2+\frac{1}{2}(\beta-\alpha))} .\nonumber \end{eqnarray}

By using the fact that
$$\Gamma(x) = \sqrt{2\pi} x^{x-\frac{1}{2}}e^{-x}[1+r(x)],\, |r(x)|\leq e^{\frac{1}{12x}}-1,\, x>0,$$
we obtain that
$$\frac{n\Gamma(n)\Gamma(2+\frac{1}{2}(\beta-\alpha))}{\Gamma(n+2+\frac{1}{2}(\beta-\alpha))} \asymp \frac{1}{n^{1+\frac{1}{2}(\beta-\alpha)}}.$$

Consequently,  it is easy to see that
$$\mu[n] \preceq \frac{1}{(n+2\theta)^{1+\frac{1}{2}(\beta-\alpha)}}$$
holds for all $n \in \mathbb{N}_0$.

By minor modifications of above arguments, we can similarly show that (\ref{mun-2}) holds if $\mu$ is a  vanishing $[1+\frac{1}{2}(\beta-\alpha)]$-Carleson measure on $[0, 1)$.  Lemma \ref{lem-1} is proved.
\end{proof}

\begin{lemma}\label{lem-2}
Let $\theta>0, -1<\alpha, \beta<1$.  For $f=\sum_{k=0}^{\infty}a_k z^k \in \mathcal{H}(\mathbb{C})$, we define
\begin{equation*}\mathbf{\check{H}}_{\theta}(f)(z):=\sum_{n=0}^\infty \left[\frac{1}{\sqrt{n!}}\sum_{k=0}^{\infty} \frac{a_k\sqrt{k!}}{(k+n+2\theta)^{1+\frac{1}{2}(\beta-\alpha)}} \right]z^n.
\end{equation*}
Then $\mathbf{\check{H}}_{\theta}$ is bounded from $\mathcal{F}^{2}_{\theta, \alpha}$ to  $\mathcal{F}^{2}_{\theta, \beta}.$
\end{lemma}

\begin{proof}
For $f=\sum_{k=0}^{\infty}{a_k}z^k \in \mathcal{F}_{\theta, \alpha}^2$, $n\in \mathbb{N}_0$, by Cauchy's inequality, we have
\begin{eqnarray}\label{g1}
\lefteqn{\left |\sum_{k=0}^{\infty}\frac{a_k\sqrt{k!}}{(k+n+2\theta)^{1+\frac{1}{2}(\beta-\alpha)}}\right| \leq \sum_{k=0}^{\infty}\frac{|a_k|\sqrt{k!}}{(k+n+2\theta)^{1+\frac{1}{2}(\beta-\alpha)}}}\nonumber \\
&=&\sum_{k=0}^{\infty} \left\{[\frac{1}{(k+n+2\theta)^{1+\frac{1}{2}(\beta-\alpha)}}]^{\frac{1}{2}}\cdot \frac{(k+\theta)^{\frac{1+\alpha}{4}}}{(n+\theta)^{\frac{1-\beta}{4}}}\cdot |a_k|\sqrt{k!}\right\}
\nonumber \\ && \quad\quad \times  \left\{[\frac{1}{(k+n+2\theta)^{1+\frac{1}{2}(\beta-\alpha)}}]^{\frac{1}{2}}\cdot \frac{(n+\theta)^{\frac{1-\beta}{4}}}{(k+\theta)^{\frac{1+\alpha}{4}}}\right\}
\nonumber \\
&\leq & \left[ \sum_{k=0}^{\infty}\frac{1}{(k+n+2\theta)^{1+\frac{1}{2}(\beta-\alpha)}}\cdot \frac{(k+\theta)^{\frac{1+\alpha}{2}}}{(n+\theta)^{\frac{1-\beta}{2}}}\cdot |a_k|^2k!\right]^{\frac{1}{2}}
\nonumber \\ && \quad\quad \times \left[\sum_{k=0}^{\infty} \frac{1}{(k+n+2\theta)^{1+\frac{1}{2}(\beta-\alpha)}}\cdot\frac{(n+\theta)^{\frac{1-\beta}{2}}}{(k+\theta)^{\frac{1+\alpha}{2}}}\right]^{\frac{1}{2}}.
\nonumber  \end{eqnarray}
Then, in view of Lemma \ref{lem-0}, we get that
\begin{eqnarray}
\lefteqn{\left |\sum_{k=0}^{\infty}\frac{a_k\sqrt{k!}}{(k+n+2\theta)^{1+\frac{1}{2}(\beta-\alpha)}}\right|}\nonumber \\
&\leq &
[w_{\alpha, \beta}^{[1]}(n)]^{\frac{1}{2}}\left[ \sum_{k=0}^{\infty}\frac{1}{(k+n+2\theta)^{1+\frac{1}{2}(\beta-\alpha)}}\cdot \frac{(k+\theta)^{\frac{1+\alpha}{2}}}{(n+\theta)^{\frac{1-\beta}{2}}}\cdot |a_k|^2k!\right]^{\frac{1}{2}}\nonumber \\
&=&
[B(\frac{1+\beta}{2}, \frac{1-\alpha}{2})]^{\frac{1}{2}} (n+\theta)^{-\frac{\beta}{2}}\nonumber \\
&&\quad\quad\quad\quad\times\left[ \sum_{k=0}^{\infty}\frac{1}{(k+n+2\theta)^{1+\frac{1}{2}(\beta-\alpha)}}\cdot \frac{(k+\theta)^{\frac{1+\alpha}{2}}}{(n+\theta)^{\frac{1-\beta}{2}}}\cdot |a_k|^2k!\right]^{\frac{1}{2}}.\nonumber \end{eqnarray}

Consequently, it follows from (\ref{w-2}) that
\begin{eqnarray}
\lefteqn{\|\mathbf{\check{H}}_{\theta}f\|_{\theta, \beta}^2= \sum_{n=0}^{\infty}(n+\theta)^{\beta}\left |\sum_{k=0}^{\infty}\frac{a_k\sqrt{k!}}{(k+n+2\theta)^{1+\frac{1}{2}(\beta-\alpha)}}\right|^2} \nonumber \\
&&\leq B(\frac{1+\beta}{2}, \frac{1-\alpha}{2})\sum_{n=0}^{\infty}\sum_{k=0}^{\infty}\frac{1}{(k+n+2\theta)^{1+\frac{1}{2}(\beta-\alpha)}}\cdot \frac{(k+\theta)^{\frac{1+\alpha}{2}}}{(n+\theta)^{\frac{1-\beta}{2}}}\cdot |a_k|^2k!
\nonumber \\
&&=B(\frac{1+\beta}{2}, \frac{1-\alpha}{2})\sum_{k=0}^{\infty}w_{\alpha, \beta}^{[2]}(k)|a_k|^{2} k!\leq [B(\frac{1+\beta}{2}, \frac{1-\alpha}{2})]^2\|f\|_{\theta, \alpha}^2.
\nonumber
\end{eqnarray}
This means that $\mathbf{\check{H}}_{\theta}: \mathcal{F}^{2}_{\theta, \alpha} \rightarrow \mathcal{F}^{2}_{\theta, \beta}$ is bounded and  Lemma \ref{lem-2} is proved.
\end{proof}

\section{Proof of Theorem \ref{main-1}}
We first show the "if" part. For $f=\sum_{k=0}^{\infty}{a_k}z^k \in \mathcal{F}_{\theta, \alpha}^2$, $n\in \mathbb{N}_0$, when $\lambda\geq 1+\frac{1}{2}(\beta-\alpha)$, we have
\begin{equation*}\label{g1}
\left |\sum_{k=0}^{\infty}\frac{a_k\sqrt{k!}}{(k+n)^{\lambda}+1}\right|\leq \sum_{k=0}^{\infty}\frac{|a_k|\sqrt{k!}}{(k+n)^{1+\frac{1}{2}(\beta-\alpha)}+1}\preceq\sum_{k=0}^{\infty}\frac{|a_k|\sqrt{k!}}{(k+n+2\theta)^{1+\frac{1}{2}(\beta-\alpha)}}. \end{equation*}

Consequently, by Lemma \ref{lem-2}, we obtain that  $\mathbf{{H}}_{\lambda}: \mathcal{F}^{2}_{\theta, \alpha} \rightarrow \mathcal{F}^{2}_{\theta, \beta}$ is bounded. This proves the "if" part of Theorem \ref{main-1}.

Next, we shall prove the "only if" part. We will show that,  if $\lambda<1+\frac{1}{2}(\beta-\alpha)$,  then, for any $\theta>0$,  $ \mathbf{H}_{\lambda}: \mathcal{F}_{\theta, \alpha}^{2} \rightarrow \mathcal{F}_{\theta, \beta}^{2}$ is not bounded.

Actually,  let $\varepsilon>0$ and set $f_{\varepsilon}=\sum_{k=0}^{\infty}a_k z^k$ with
$$a_0=0,\, a_k=\sqrt{\varepsilon\theta^\varepsilon}(k+\theta)^{-\frac{\alpha+1+\varepsilon}{2}}\frac{1}{\sqrt{k!}},\, k\in \mathbb{N}.$$
It is easy to see that $$\|f_\varepsilon\|_{\theta, \alpha}^2=\varepsilon\theta^\varepsilon\sum_{k=1}^{\infty}(n+\theta)^{-1-\varepsilon}\leq  \varepsilon\theta^\varepsilon  \int_{0}^{\infty}(x+\theta)^{-1-\varepsilon}dx=1.$$

Then we have
\begin{eqnarray}\label{eq--1}\|\mathbf{H}_{\lambda}f_\varepsilon\|_{\theta,\beta}^2&=&\varepsilon\theta^\varepsilon\sum_{n=0}^{\infty}(n+\theta)^{\beta}\left|\sum_{k=1}^{\infty}\frac{1}{(k+n)^{\lambda}+1}\cdot (k+\theta)^{-\frac{\alpha+1+\varepsilon}{2}}\right|^2 \\
&\succeq &\varepsilon\theta^\varepsilon\sum_{n=0}^{\infty}(n+\theta)^{\beta}\left|\sum_{k=1}^{\infty}\frac{1}{(k+n+2\theta)^{\lambda}}\cdot (k+\theta)^{-\frac{\alpha+1+\varepsilon}{2}}\right|^2 \nonumber
\end{eqnarray}
On the other hand, we notice that
\begin{eqnarray}\label{eq--2}\lefteqn{ \sum_{k=1}^{\infty}\frac{1}{(k+n+2\theta)^{\lambda}}\cdot (k+\theta)^{-\frac{\alpha+1+\varepsilon}{2}}} \\
&\geq&
\int_{1}^{\infty} \frac{1}{(x+\theta+n+\theta)^{\lambda}}\cdot  (x+\theta)^{-\frac{\alpha+1+\varepsilon}{2}}\,dx\nonumber  \\
&=&
\int_{1+\theta}^{\infty} \frac{1}{(s+n+\theta)^{\lambda}}\cdot  s^{-\frac{\alpha+1+\varepsilon}{2}}\,ds\nonumber  \\
&=&
(n+\alpha)^{1-\lambda-\frac{\alpha+1+\varepsilon}{2}}\int_{\frac{1+\theta}{n+\theta}}^{\infty} \frac{1}{(1+t)^{\lambda}}\cdot  t^{-\frac{\alpha+1+\varepsilon}{2}}\,dt
\nonumber
\end{eqnarray}

Combine (\ref{eq--1}) and (\ref{eq--2}), we get that
\begin{eqnarray}\|\mathbf{H}_{\lambda}f_\varepsilon\|_{\theta,\beta}^2&\succeq&\varepsilon\theta^\varepsilon\sum_{n=0}^{\infty}(n+\theta)^{(\beta-\alpha)+2(1-\lambda)-1-\varepsilon}\int_{\frac{1+\theta}{n+\theta}}^{\infty} \frac{1}{(1+t)^{\lambda}}\cdot  t^{-\frac{\alpha+1+\varepsilon}{2}}\,dt\nonumber \\
&\geq &\varepsilon\theta^\varepsilon\sum_{n=0}^{\infty}(n+\theta)^{(\beta-\alpha)+2(1-\lambda)-1-\varepsilon}\int_{1+\frac{1}{\theta}}^{\infty} \frac{1}{(1+t)^{\lambda}}\cdot  t^{-\frac{\alpha+1+\varepsilon}{2}}\,dt \nonumber
\end{eqnarray}

If $\lambda< 1+\frac{1}{2}(\beta-\alpha)$, we see that $(\beta-\alpha)+2(1-\lambda)>0$. We suppose that $\mathbf{H}_{\lambda}: \mathcal{F}_{\theta, \alpha}^2 \rightarrow  \mathcal{F}_{\theta, \beta}^2$ is bounded, then there exists a constant $C_4>0$ such that
\begin{eqnarray}\label{c-1}C_4&\geq&\frac{\|\mathbf{H}_{\lambda}f_{\varepsilon}\|_{\theta, \beta}^2}{\|f_{\varepsilon}\|_{\theta,\alpha}^2}\\
&\geq& \varepsilon\theta^\varepsilon\sum_{n=0}^{\infty}(n+\theta)^{(\beta-\alpha)+2(1-\lambda)-1-\varepsilon}\int_{1+\frac{1}{\theta}}^{\infty} \frac{1}{(1+t)^{\lambda}}\cdot  t^{-\frac{\alpha+1+\varepsilon}{2}}\,dt. \nonumber
\end{eqnarray}

But when $\varepsilon<(\beta-\alpha)+2(1-\lambda)$, we have $$\sum_{n=0}^{\infty}(n+\theta)^{(\beta-\alpha)+2(1-\lambda)-1-\varepsilon}=\infty.$$
Hence we get that (\ref{c-1}) is a contradiction. This implies that $\mathbf{H}_{\lambda}: \mathcal{F}_{\theta, \alpha}^2 \rightarrow  \mathcal{F}_{\theta, \beta}^2$ can not be bounded when $\lambda< 1+\frac{1}{2}(\beta-\alpha)$. This proves the "only if" part of Theorem \ref{main-1}.

The Theorem \ref{main-1} is now proved.

\section{Proof of Theorem \ref{main-2}}
\begin{proof}[Proof of "if" part of Theorem \ref{main-2}]
By combining  Lemma \ref{lem-1} and Lemma \ref{lem-2}, we easily see that, for any $\theta>0$,   $\mathbf{H}_{\mu}: \mathcal{F}_{\theta, \alpha}^{2}\rightarrow \mathcal{F}_{\theta, \beta}^{2}$ is bounded if $\mu$ is a $[1+\frac{1}{2}(\beta-\alpha)]$-Carleson measure on $[0, 1)$. The "if" part of Theorem \ref{main-2} is proved.
\end{proof}

\begin{proof}[Proof of "only if" part of Theorem \ref{main-2}]
In our proof, we need the following well-known estimate, see [18, Page 54].  Let $0<w<1$. For any $c>0$, we have
\begin{equation}\label{est}\sum_{n=1}^{\infty}n^{c-1}w^{2n}\asymp \frac{1}{(1-w^2)^c}.
\end{equation}

For any $0<w<1$, we set ${f}_w=\sum_{k=0}^{\infty} {a}_k z^k \in \mathcal{H}(\mathbb{C})$ with
$${a}_k=(1-w^2)^{\frac{1}{2}}(k+\theta)^{-\frac{\alpha}{2}}w^{k}\frac{1}{\sqrt{k!}},\; k\in \mathbb{N}_0.$$ Then we see from (\ref{est}) that $\|{f}_w\|_{\theta, \alpha}^2\asymp 1.$

In view of the boundedness of $\mathbf{H}_{\mu}: \mathcal{F}_{\theta, \alpha}^{2}\rightarrow \mathcal{F}_{\theta, \beta}^{2}$, we obtain that
\begin{eqnarray}\label{boun-1}
1 &\succeq & \|\mathbf{H}_{\mu}{f}_w\|_{\theta, \beta}^2 \\
&=&\sum_{n=0}^{\infty}(n+\theta)^{\beta}\left [\sum_{k=0}^{\infty}{a}_k\sqrt{k!}\int_{0}^1t^{k+n}d\mu(t) \right]^2 \nonumber \\
&=&(1-w^2)\sum_{n=0}^{\infty}(n+\theta)^{\beta}\left[\sum_{k=0}^{\infty}(k+\theta)^{-\frac{\alpha}{2}}w^{k}\int_{0}^{1}t^{k+n}d\mu(t)\right]^{{2}}
\nonumber \\
&\succeq  &(1-w^2)\sum_{n=1}^{\infty}n^{\beta}\left[\sum_{k=1}^{\infty}k^{-\frac{\alpha}{2}}w^{k}\int_{w}^{1}t^{k+n}d\mu(t)\right]^{{2}}.\nonumber
\end{eqnarray}

On the other hand, by again (\ref{est}), we have
\begin{eqnarray}\label{boun-2}
\lefteqn{\sum_{n=1}^{\infty}n^{\beta}\left[\sum_{k=1}^{\infty}k^{-\frac{\alpha}{2}}w^{k}\int_{w}^{1}t^{k+n}d\mu(t)\right]^{{2}} }\\
&\geq &[\mu([w, 1))]^{2}\sum_{n=1}^{\infty}n^{\beta}\left[\sum_{k=1}^{\infty}k^{-\frac{\alpha}{2}}w^{k}\cdot w^{k+n}\right]^{{2}}\nonumber \\
&=&[\mu([w, 1))]^{2}\left[\sum_{n=1}^{\infty}n^{\beta}w^{2n}\right]\left[\sum_{k=1}^{\infty}k^{-\frac{\alpha}{2}}w^{2k}\right]^{{2}}\nonumber \\
&\asymp &[\mu([w, 1))]^{2}\cdot \frac{1}{(1-w^2)^{1+\beta}} \cdot\frac{1}{(1-w^2)^{2-\alpha}}.
\nonumber
\end{eqnarray}

We see from (\ref{boun-1}) and (\ref{boun-2}) that
$$\mu([w, 1))\preceq (1-w^2)^{1+\frac{1}{2}(\beta-\alpha)}$$
holds for all $0<w<1$. It follows that $\mu$ is a $[1+\frac{1}{2}(\beta-\alpha)]$-Carleson measure on $[0, 1)$ and the "only if" part is proved.
\end{proof}
Now, the proof of Theorem \ref{main-2} is finished.

\section{Proof of Theorem \ref{main-3}}

We first show the "if" part.  For any $f=\sum_{k=0}^{\infty}a_k z^k \in \mathcal{F}_{\theta, \alpha}^2$. Let $\mathfrak{N}\in \mathbb{N}$, we consider
\begin{equation*} \mathbf{H}_{\mu}^{[\mathfrak{N}]}(f)(z):=\sum_{n=0}^{\mathfrak{N}} \left[\frac{1}{\sqrt{n!}}\sum_{k=0}^{\infty}\mu[k+n]a_k\sqrt{k!}\right]z^n,  z\in \mathbb{C}.
\end{equation*}
Then we see that $ \mathbf{H}_{\mu}^{[\mathfrak{N}]}$ is a finite rank operator and hence it is compact from $\mathcal{F}_{\theta, \alpha}^2$ to $\mathcal{F}_{\theta, \beta}^2$.

By Lemma \ref{lem-1}, we see that, for any $\epsilon>0$, there is an $N_0 \in \mathbb{N}$ such that
$$\mu[n] \preceq \frac{\epsilon}{(n+2\theta)^{1+\frac{1}{2}(\beta-\alpha)}}$$
holds for all $n>N_0$.

Note that
\begin{eqnarray}\|(\mathbf{H}_{\mu}-\mathbf{H}_{\mu}^{[\mathfrak{N}]})(f)\|_{\theta, \beta}^2=\sum_{n=\mathfrak{N}+1}^{\infty}(n+\theta)^{\beta}  \left|\frac{1}{\sqrt{n!}}\sum_{k=0}^{\infty}\mu[k+n]a_k\sqrt{k!}\right|^2.\nonumber
\end{eqnarray}
When $\mathfrak{N}>N_0$, we get that
\begin{eqnarray}\|(\mathbf{H}_{\mu}-\mathbf{H}_{\mu}^{[\mathfrak{N}]})(f)\|_{\theta, \beta}^2\preceq
 {\epsilon}^2 \sum_{n=\mathfrak{N}+1}^{\infty}(n+\theta)^{\beta} \left|\sum_{k=0}^{\infty}\frac{a_k}{(k+n+2\theta)^{1+\frac{1}{2}(\beta-\alpha)}} \right|^2.\nonumber\end{eqnarray}

Consequently, by Lemma \ref{lem-2}, we see that, for any $\epsilon>0$, there is an $N_0 \in \mathbb{N}$ such that
\begin{eqnarray}\|(\mathbf{H}_{\mu}-\mathbf{H}_{\mu}^{[\mathfrak{N}]})(f)\|_{\theta, \beta}^2\preceq
 {\epsilon}^2  \|f\|^2_{\theta, \alpha}\nonumber\end{eqnarray}
holds for all $\mathfrak{N}>N_0$. We conclude that $\mathbf{H}_{\mu}$ is compact from $\mathcal{F}_{\theta, \alpha}^2$ to $\mathcal{F}_{\theta, \beta}^2$. This proves the "if" part.

Next, we show the "only if" part.  For $0<w<1$.  We set  $\widetilde{f}_w=\sum_{k=0}^{\infty}\widetilde{a}_k z^k$ with
$$\widetilde{a}_k=(1-w^2)^{\frac{1+\alpha}{2}}w^{k} \frac{1}{\sqrt{k!}},\: k\in \mathbb{N}_0.
$$

We see from (\ref{est}) that $\|\widetilde{f}_w\|_{\theta, \alpha}\asymp1$, and we conclude that $\widetilde{f}_w$ is convergent weakly to $0$ in $\mathcal{F}_{\theta, \alpha}^2$ as $w\rightarrow 1^{-}$. Since $\mathbf{H}_{\mu}$ is compact from $\mathcal{F}_{\theta, \alpha}^2$ to $\mathcal{F}_{\theta, \beta}^2$, we get
\begin{equation}\label{com-0}\lim_{w\rightarrow {1^{-}}} \|\mathbf{H}_{\mu}(\widetilde{f}_w)\|_{\theta, \beta}=0.\end{equation}

On the other hand, we have
\begin{eqnarray}
 \|\mathbf{H}_{\mu}(\widetilde{f}_w)\|_{\theta, \beta}^2
&=&(1-w^2)^{1+\alpha}\sum_{n=0}^{\infty}(n+\theta)^{\beta}\left[\sum_{k=0}^{\infty}w^{k}\int_{0}^{1}t^{k+n}d\mu(t)\right]^{{2}}\nonumber \\
&\succeq &(1-w^2)^{1+\alpha}\sum_{n=1}^{\infty}n^{\beta}\left[\sum_{k=1}^{\infty}w^{k}\int_{w}^{1}t^{k+n}d\mu(t)\right]^{{2}}\nonumber \\
&\geq &(1-w^2)^{1+\alpha}[\mu([w, 1))]^{2}\left[\sum_{n=1}^{\infty}n^{\beta}w^{2n}\right]\left[\sum_{k=1}^{\infty}w^{2k}\right]^{{2}}\nonumber \\
&\succeq &(1-w^2)^{1+\alpha}[\mu([w, 1))]^{2}\cdot \frac{1}{(1-w^2)^{\beta+1}}\cdot \frac{1}{(1-w^2)^2}.\nonumber
\end{eqnarray}
Then we get that
\begin{eqnarray}
 \mu([w, 1))\preceq \|\mathbf{H}_{\mu}(\widetilde{f}_w)\|_{\theta, \beta}(1-w^2)^{1+\frac{1}{2}(\beta-\alpha)}\nonumber.
\end{eqnarray}

It follows from (\ref{com-0}) that $\mu$ is a vanishing $[1+\frac{1}{2}(\beta-\alpha)$-Carleson measure on $[0, 1)$. This proves the "only if" part of Theorem \ref{main-3}  and the proof of Theorem \ref{main-3} is completed.

\section{Final remarks}

\begin{remark}

For $\lambda>0$, we define the {\em Hilbert-type operator}  $\mathbf{\widehat{H}}_{\lambda}(f)$ as
\begin{equation*} \mathbf{\widehat{H}}_{\lambda}(f)(z):=\sum_{n=0}^\infty \left(\sum_{k=0}^{\infty} \frac{a_k}{(k+n+1)^{\lambda}} \right)z^n.
\end{equation*}
When $\lambda=1$, we get the classical Hilbert operator. We next will study the boundedness of $\mathbf{\widehat{H}}_{\lambda}(f)$ acting on certain spaces of  holomorphic functions on $\mathbb{D}$.

Let $p$ be a positive number and  $X_p$ be a Banach space of holomorphic functions on $\mathbb{D}$. For any $f\in X_p$, we assume that the norm $\|f\|_{X_{p}}$  of $f$ is determined by  $f, p$ and other finite parameters $\beta_1, \beta_2, \cdots, \beta_m$; $m\in \mathbb{N}_0$($m=0$ means that there is no parameter).

For any $f=\sum_{n=0}^{\infty}a_n z^n\in \mathcal{H}(\mathbb{D})$ with $\{a_n\}_{n=0}^{\infty}$ is a decreasing sequence of non-negative real numbers,  we say $X_p$ have the {\em property $(\star)$}, if
there exist a function $\mathbf{G}_X=\mathbf{G}_{X}(p, \beta_1, \beta_2, \cdots, \beta_m)$ with $\mathbf{G}_X>-1$ such that  $f \in  X_p$  if and only if
$$\sum_{n=1}^{\infty} n^{\mathbf{G}_X}a_n^{p}<\infty.$$

We list several classical spaces of holomorphic functions on $\mathbb{D}$, which have the property $(\star)$.   Let  $f=\sum_{n=0}^{\infty}a_n z^n\in \mathcal{H}(\mathbb{D})$ with $\{a_n\}_{n=0}^{\infty}$ is a decreasing sequence of non-negative real numbers.  For example,

(1) the Hardy space $H^p(\mathbb{D}), 1<p<\infty$, we know that  $f \in H^p(\mathbb{D})$ if and only if $$\sum_{n=1}^{\infty} n^{p-2}a_n^{p}<\infty.$$ See \cite{Pa}.

(2) For $1<p<\infty$,  let $p-2<\alpha\leq p-1$.  It holds that $f \in D^p_{\alpha}(\mathbb{D})$ if and only if
 $$\sum_{n=1}^{\infty} n^{2p-3-\alpha}a_n^{p}<\infty.$$
 Here $D^p_{\alpha}(\mathbb{D})$ is the Dirichlet-type space, defined as
 $$D^p_{\alpha}(\mathbb{D})=\left\{f\in \mathcal{H}({\mathbb{D}}): \|f\|_{D^p_{\alpha}}=|f(0)|+\left[\int_{\mathbb{D}}|f'(z)|^p(1-|z|^2)^{\alpha}dA(z)\right]^{\frac{1}{p}}<\infty\right\}.$$
 See \cite{GM-2}.

 (3) For $1<p<\infty$, when $-1<\alpha< p-2$, we have $f \in A^p_{\alpha}(\mathbb{D})$ if and only if
 $$\sum_{n=1}^{\infty} n^{2p-3-\alpha}a_n^{p}<\infty.$$
 Here $A^p_{\alpha}(\mathbb{D})$ is the Bergman space, defined as
 $$A^p_{\alpha}(\mathbb{D})=\left\{f\in \mathcal{H}({\mathbb{D}}): \|f\|_{A^p_{\alpha}}=\left[(\alpha+1)\int_{\mathbb{D}}|f(z)|^p(1-|z|^2)^{\alpha}dA(z)\right]^{\frac{1}{p}}<\infty\right\}.$$
 See \cite{GM-2}.

 We obtain that

\begin{proposition}\label{rem}
Let $\lambda$, $p$ be two positive numbers, $X_p$ be a Banach space of holomorphic functions on $\mathbb{D}$ having the property $(\star)$ and  $\mathbf{\widehat{H}}_{\lambda}$ be as above. Then the necessary condition of $\mathbf{\widehat{H}}_{\lambda}: X_p\rightarrow X_p$ is bounded  is $\lambda\geq 1$.
\end{proposition}

\begin{proof}
It is enough to prove that,  $\mathbf{\widehat{H}}_{\lambda}: X_p\rightarrow X_p$  can not be bounded, if $0<\lambda<1$.

Let $\varepsilon>0$ and set $\widehat{f}_{\varepsilon}=\sum_{k=0}^{\infty}\widehat{a}_k z^k$ with
$\widehat{a}_0=(\frac{\varepsilon}{1+\varepsilon})^{\frac{1}{p}},\, \widehat{a}_k=(\frac{\varepsilon}{1+\varepsilon})^{\frac{1}{p}}k^{-\frac{\mathbf{G}_X+1+\varepsilon}{p}},\, k\geq 1.$
It is easy to see that $\{\widehat{a}_k\}_{k=0}^{\infty}$ is a decreasing sequence and $\sum_{k=1}^{\infty}k^{\mathbf{G}_{X}}\widehat{a}_k^{p}<\infty.$ Hence $\widehat{f}_\varepsilon \in X_p$.

Set
$$b_n=\sum_{k=0}^{\infty}\frac{\widehat{a}_k}{(k+n+1)^{\lambda}}, \, n\in \mathbb{N}_0.$$

We suppose that $\mathbf{\widehat{H}}_{\lambda}: X_p\rightarrow X_p$  is bounded. Then, by the fact that $\{b_n\}_{n=0}^{\infty}$ is a decreasing sequence, we see that $\widehat{f}=\sum_{n=0}^{\infty}b_n z^n \in X_p$ and hence
$$\sum_{n=1}^{\infty} n^{\mathbf{G}_X} b_n^p <\infty.$$
Then
\begin{eqnarray}\label{e-1}
 \infty &>& \sum_{n=1}^{\infty} n^{\mathbf{G}_X} \left[\sum_{k=0}^{\infty}\frac{\widehat{a}_k}{(k+n+1)^{\lambda}}\right]^p  \\ &\geq&  \frac{\varepsilon}{1+\varepsilon}  \sum_{n=1}^{\infty} n^{\mathbf{G}_X} \left[\sum_{k=1}^{\infty}\frac{1}{(k+n+1)^{\lambda}} \cdot  k^{-\frac{\mathbf{G}_X+1+\varepsilon}{p}}\right]^p\nonumber
\end{eqnarray}

On the other hand, from the fact that $k+n+1\leq 2(k+n)$ for all $k ,n \in \mathbb{N}$, we see that
\begin{eqnarray}\label{e-2}
\lefteqn{\sum_{k=1}^{\infty}\frac{1}{(k+n+1)^{\lambda}} \cdot  k^{-\frac{\mathbf{G}_X+1+\varepsilon}{p}}}\\
&\geq& \frac{1}{2^{\lambda}}\sum_{k=1}^{\infty}\frac{1}{(k+n)^{\lambda}} \cdot  k^{-\frac{\mathbf{G}_X+1+\varepsilon}{p} }\nonumber \\
&\geq& \frac{1}{2^{\lambda}}\int_{1}^{\infty} \frac{1}{(x+n)^{\lambda}} \cdot  x^{-\frac{\mathbf{G}_X+1+\varepsilon}{p} }\,dx \nonumber \\ & =&
n^{(1-\lambda)-\frac{\mathbf{G}_X+1+\varepsilon}{p}} \cdot \frac{1}{2^{\lambda}}\int_{\frac{1}{n}}^{\infty} \frac{1}{(1+s)^{\lambda}}\cdot s^{-\frac{\mathbf{G}_X+1+\varepsilon}{p} }\,ds \nonumber \\
&\geq& n^{(1-\lambda)-\frac{\mathbf{G}_X+1+\varepsilon}{p}}\cdot \frac{1}{2^{\lambda}} \int_{1}^{\infty} \frac{1}{(1+s)^{\lambda}}\cdot s^{-\frac{\mathbf{G}_X+1+\varepsilon}{p} }\,ds \nonumber
\end{eqnarray}

Combine (\ref{e-1}) and (\ref{e-2}), we get that
\begin{eqnarray}\label{e-3}
\infty &>& \frac{\varepsilon}{1+\varepsilon} \left[\sum_{n=1}^{\infty}n^{p(1-\lambda)-1-\varepsilon}\right]\cdot \left[\frac{1}{2^{\lambda}} \int_{1}^{\infty} \frac{1}{(1+s)^{\lambda}}\cdot s^{-\frac{\mathbf{G}_X+1+\varepsilon}{p} }\,ds\right]^p.\end{eqnarray}

If $\lambda< 1$,  when $\varepsilon<p(1-\lambda)$, we have $\sum_{n=1}^{\infty}n^{p(1-\lambda)-1-\varepsilon}=\infty.$ Thus (\ref{e-3}) is a contradiction. This means that $\mathbf{\widehat{H}}_{\lambda}: X_p\rightarrow X_p$  can not be bounded if $\lambda< 1$.  The Theorem \ref{rem} is proved.
\end{proof}
\end{remark}

\begin{remark}

From Theorem \ref{main-1},  we know that  $ \mathbf{H}_{\lambda}$ is not bounded from $\mathcal{F}_{\theta, \alpha}^{2}$ to $\mathcal{F}_{\theta, \beta}^{2}$ if $0<\lambda <1+\frac{1}{2}(\beta-\alpha)$.   On the other hand, we notice that, for $\lambda>0$, it holds that
$$\frac{1}{(k+n)^{\lambda}+1}\asymp \frac{1}{(k+n+1)^{\lambda}}\asymp\int_{0}^1t^{k+n}(1-t)^{\lambda-1}dt$$
for all  $k, n\in \mathbb{N}_0$.  To make $ \mathbf{H}_{\lambda}$ bounded from $\mathcal{F}_{\theta, \alpha}^{2}$ to $\mathcal{F}_{\theta, \beta}^{2}$ when $0<\lambda <1+\frac{1}{2}(\beta-\alpha)$.  For $\lambda>0, f=\sum_{n=0}^{\infty}a_n z^n \in \mathcal{H}(\mathbb{C})$,  let $\mu$ be a positive Bore measure on $[0, 1)$,  we define a new  {\em Hilbert-type operator $\mathbf{H}_{\lambda}^{\mu}$} as
\begin{equation*} \mathbf{H}_{\lambda}^{\mu}(f)(z):=\sum_{n=0}^\infty \left(\frac{1}{\sqrt{n!}}\sum_{k=0}^{\infty}\mu_{\lambda}[k+n] a_k\sqrt{k!}\right)z^n.
\end{equation*}
Here
\begin{equation*}\mu_{\lambda}{[k+n]}=\int_{0}^1 t^{k+n}(1-t)^{\lambda-1}d\mu(t),\: k, n\in \mathbb{N}_{0}.\end{equation*}

In view of the proofs of Theorem \ref{main-2} and \ref{main-3}, we can get the following result, which is a generalization of Theorem \ref{main-2} and \ref{main-3}.
\begin{proposition}
Let $\lambda>0, -1<\alpha, \beta<1$.  Let $\mu$ be a positive Bore measure on $[0, 1)$ such that $d\nu(t)=(1-t)^{\lambda-1}d\mu(t)$ is a finite measure on $[0, 1)$, and $\mathbf{H}_{\lambda}^{\mu}$ be as above. Then, for any $\theta>0$,  it holds that

(1) $\mathbf{H}_{\lambda}^{\mu}: \mathcal{F}_{\theta, \alpha}^{2}\rightarrow \mathcal{F}_{\theta, \beta}^{2}$ is bounded if and only if $\nu$ is a $[1+\frac{1}{2}(\beta-\alpha)]$-Carleson measure on $[0, 1)$;

(2) $\mathbf{H}_{\lambda}^{\mu}: \mathcal{F}_{\theta, \alpha}^{2}\rightarrow \mathcal{F}_{\theta, \beta}^{2}$ is compact if and only if $\nu$ is a vanishing $[1+\frac{1}{2}(\beta-\alpha)]$-Carleson measure on $[0, 1)$.
\end{proposition}
\end{remark}

\section*{Acknowledgement}
We thank the referee for invaluable comments and useful suggestions on this paper.


\begin{thebibliography}{99}

\bibitem{BW} Bao G., Wulan H.,  {\em Hankel matrices acting on Dirichlet spaces}, J. Math. Anal. Appl.,  409(2014), pp.  228-235.
\bibitem{BK} Bo\v{z}in V., Karapetrovi\'{c}, B., {\em  Norm of the Hilbert matrix on Bergman spaces},  J. Funct. Anal., 274 (2018), no. 2, pp. 525-543.

\bibitem{CGP}Chatzifountas C., Girela D., Pel\'{a}ez J.,  {\em A generalized Hilbert matrix acting on Hardy spaces}, J. Math. Anal. Appl., 413(2014), pp. 154-168.
\bibitem{CZ}Cho H., Zhu K.,  {\em Fock–Sobolev spaces and their Carleson measures}, Journal of Functional Analysis, 263 (2012),  pp. 2483-2506.

\bibitem{D-1} Diamantopoulos E., {\em Operators induced by Hankel matrices on Dirichlet spaces}, Analysis (Munich), 24 (2004), no. 4, pp. 345-360.
\bibitem{D}Diamantopoulos E., {\em Hilbert matrix on Bergman spaces}. Illinois J. Math., 48 (2004), no. 3, pp. 1067-1078.


\bibitem{DS}Diamantopoulos E., Siskakis  Aristomenis G., {\em Composition operators and the Hilbert matrix}, Studia Math., 140 (2000), no. 2, pp. 191-198.

\bibitem{DJV} Dostani\'{c} M., Jevti\'{c}, M., Vukoti\'{c}, D., {\em  Norm of the Hilbert matrix on Bergman and Hardy spaces and a theorem of Nehari type}, J. Funct. Anal., 254 (2008), no. 11, pp. 2800-2815.

\bibitem{GGPS}Galanopoulos P.,  Girela D., Pel\'{a}ez  J.,  Siskakis  Aristomenis G., {\em Generalized Hilbert operators},  Ann. Acad. Sci. Fenn. Math., 39 (2014), no. 1, pp. 231-258.
\bibitem{GP}Galanopoulos P.,  Pel\'{a}ez J.,  {\em A Hankel matrix acting on Hardy and Bergman spaces}, Studia Math., 200(2010), pp. 201-220.

\bibitem{GM}Girela D., Merch\'{a}n N., {\em A generalized Hilbert operator acting on conformally invariant spaces}, Banach Journal of Mathematical Analysis, 12(2018), pp. 374-398.
\bibitem{GM-2}Girela D.,  Merch\'{a}n N., {\em Hankel matrices acting on the Hardy space $H^1$ and on Dirichlet spaces}, Revista Matematica Complutense,  32(2019), pp. 799-822.







\bibitem{Ka} Karapetrovi\'{c} B., {\em Hilbert matrix and its norm on weighted Bergman spaces}, J. Geom. Anal.,  31 (2021), no. 6, pp. 5909-5940.
\bibitem{LS} Li S., Stevi\'{c} S., {\em Generalized Hilbert operator and Fejér-Riesz type inequalities on the polydisc}, Acta Math. Sci. Ser. B (Engl. Ed.),  29(2009),  no. 1,  pp. 191-200.
\bibitem{PR1}Paulsen V.,  Raghupathi  M.,  {\em An Introduction to the Theory of Reproducing Kernel Hilbert Spaces},  Cambridge University Press, Cambridge, 2016.

\bibitem{Pa} Pavlovi\'{c} M., {\em Introduction to function spaces on the disk},  Matemati\v{c}ki Institut SANU, Belgrade, 2004.




\bibitem{PR}Pel\'{a}ez J.,  R\"{a}tty\"{a} J., {\em Generalized Hilbert operators on weighted Bergman spaces}. Adv. Math.,  240 (2013),  pp. 227-267.


\bibitem{Sh} Shields A., {\em Weighted shift operators and analytic function theory}, in Topics in Operator Theory, Math. Surveys, No. 13, Amer. Math. Soc., Providence, 1974.

\bibitem{W} Wang Z.,  Gua D.,  {\em An Introduction to Special Functions}, Science Press, Beijing, 1979.

\bibitem{YZ1}Ye S., Zhou, Z., {\em A derivative-Hilbert operator acting on the Bloch space}, Complex Anal. Oper. Theory,  15 (2021), no. 5, Paper No. 88, 16 pp.
\bibitem{YZ2}Ye S., Zhou, Z.,  {\em A Derivative-Hilbert operator acting on Bergman spaces}, J. Math. Anal. Appl., 506(2022), 125553, pp. 1-18.
\bibitem{Zh}Zhu K., {\em Operator Theory in Function Spaces}, Marcel Dekker, New York, 1990.
\bibitem{Zhu} Zhu K., {\em Analysis on Fock spaces}, Springer, New York,  2012.



\end{thebibliography}
\end{document}